\documentclass[12pt]{amsart}
\usepackage{amsfonts,amssymb,amscd,amstext,mathrsfs,color}

\input xy
\xyoption{all}
 
\newtheorem{theorem}{Theorem}
\newtheorem{corollary}{Corollary}
\newtheorem{lemma}{Lemma}
\newtheorem{proposition}{Proposition}

\theoremstyle{definition}
\newtheorem{remark}{Remark}

\newcommand{\C}{\mathbb{C}}

\newcommand{\End}{{\operatorname{End}}}
\newcommand{\Aut}{{\operatorname{Aut}}}
\newcommand{\id}{{\operatorname{id}}}
\newcommand{\hyp}{{\operatorname{hyp}}}
\newcommand{\att}{{\operatorname{att}}}
\newcommand{\sad}{{\operatorname{sad}}}
\newcommand{\rep}{{\operatorname{rep}}}

\newcommand{\rne}{{\operatorname{rne}}}
\newcommand{\tam}{{\operatorname{tam}}}
\newcommand{\bas}{{\operatorname{bas}}}
\newcommand{\nrc}{{\operatorname{nrc}}}
\newcommand{\npr}{{\operatorname{npr}}}

\begin{document}

\title{Dynamics of generic automorphisms \\ of Stein manifolds with the density property}

\author{Leandro Arosio and Finnur L\'arusson}

\address{Dipartimento Di Matematica, Universit\`a di Roma \lq\lq Tor Vergata\rq\rq, Via Della Ricerca Scientifica 1, 00133 Roma, Italy}
\email{arosio@mat.uniroma2.it}

\address{School of Mathematical Sciences, University of Adelaide, Adelaide SA 5005, Australia}
\email{finnur.larusson@adelaide.edu.au}


\thanks{L.~Arosio was partially supported  by  PRIN {\sl  Real and Complex Manifolds: Geometry and Holomorphic Dynamics} n. 2022AP8HZ9, by INdAM, and by the   MUR Excellence Department Project MatMod@TOV
CUP:E83C23000330006}

\date{7 March 2023.  Minor edits 21 March 2025}


\begin{abstract}
We study the dynamics of a generic automorphism $f$ of a Stein manifold with the density property.  Such manifolds include almost all linear algebraic groups.  Even in the special case of $\C^n$, $n\geq 2$, most of our results are new.  We study the Julia set, non-wandering set, and chain-recurrent set of $f$.  We show that the closure of the set of saddle periodic points of $f$ is the largest forward invariant set on which $f$ is chaotic.  This subset of the Julia set of $f$ is also characterised as the closure of the set of transverse homoclinic points of $f$, and equals the Julia set if and only if a certain closing lemma holds.  Among the other results in the paper is a generalisation of Buzzard's holomorphic Kupka-Smale theorem to our setting.  
\end{abstract}

\maketitle

\tableofcontents

\section{Introduction and main results} 
\label{sec:intro}

\noindent
This paper continues the program of research, built on the groundbreaking work of Forn\ae ss and Sibony in \cite{FS1997}, that began with our previous papers \cite{AL2019, AL2020, AL2022}.  Here, we investigate the dynamics of a generic automorphism of a Stein manifold $X$ with the density property.  Genericity, of course, is with respect to the compact-open topology on the automorphism group $\Aut\, X$, which is separable and defined by a complete metric.  Roughly speaking, a Stein manifold has the density property if it has many complete holomorphic vector fields and hence, by integrating such fields, many automorphisms.  For an overview of Anders\'en-Lempert theory, which is the theory of Stein manifolds with the density property or one of its variants, see \cite[Chapter 4]{Forstneric2017} or \cite{FK2021}.  The density property turns out to naturally fit into dynamical arguments, allowing us to develop a rich theory of reversible holomorphic dynamics.  The prototypical example of a Stein manifold with the density property is $\C^n$, $n\geq 2$.  For more examples, see Remark \ref{r:main-remark}(1) below.

The ideal picture that we work towards, but do not conjecture because it may be too simplistic, is that a generic automorphism $f$ of $X$ is chaotic on the complement of the union of the basins of its attracting and repelling\footnote{By the {\it basin} of a repelling cycle we mean the basin of attraction of the inverse map.} cycles.  Using our previous work on closing lemmas in \cite{AL2020}, we identify the largest subset of $X$ on which $f$ is chaotic.  This subset is defined as the closure of the set of saddle periodic points of $f$.  It is also characterised as the closure of the set of transverse homoclinic points of $f$.  We call it the {\it chaotic Julia set} of $f$ and denote it by $J_f^*$.

The following theorem contains the main results of this paper.  Even in the special case of $\C^n$, $n\geq 2$, the theorem is new, except for part (a), which is a generalisation of Buzzard's holomorphic Kupka-Smale theorem for $\C^n$, $n\geq 2$ \cite{Buzzard1998}.  A key ingredient in the proof is Theorem \ref{t:stable-unstable}, which builds on the work of Peters, Vivas, and Wold in \cite{PVW2008}.  The notation used below is explained at the end of the introduction.

\begin{theorem}  \label{t:main-theorem}
A generic automorphism $f$ of a Stein manifold $X$ with the density property has the following properties.
\begin{enumerate}
\item[(a)]  Every periodic point of $f$ is hyperbolic and every homoclinic or heteroclinic point is transverse.
\item[(b)]  The forward Julia set $J_f^+$ is connected, has empty interior, is the boundary of each connected component of the basin of attraction of every attracting cycle of $f$, and is not an embedded topological manifold at any of its points.  Also, $f$ has a saddle fixed point $q$ such that $J_f^+$ is the closure of the stable manifold $W_f^s(q)$ of $f$ through $q$.
\item[(c)]  The forward Fatou set $F_f^+=X\setminus J_f^+$ equals the subset $\rne(f)$ on which $f$ is robustly non-expelling and is the union of the basins of attraction of the attracting cycles of $f$ and the non-recurrent Fatou components of $f$.  Every connected component of $F_f^+$ is Stein.
\item[(d)]  The non-wandering set $\Omega_f$ and the Julia set $J_f$ are not compact (hence not empty), have empty interior, and
\[ \Omega_f = J_f \cup \att(f) \cup \rep(f). \]
\item[(e)]  The set $C_f$ of chain-recurrent points is the complement of the union of  the proper basins of the attracting and repelling cycles of $f$.
The chain-recurrence classes are the attracting cycles, the repelling cycles, and the set
\[ \big(J_f^+ \cup \nrc(f)\big) \cap \big(J_f^- \cup \nrc(f^{-1})\big), \]
where $\nrc(f)$ denotes the union of the non-recurrent Fatou components of $f$.
\item[(f)]  The chaotic Julia set $J_f^*=\overline{\sad(f)}\subset J_f$ is perfect and not compact, and $f$ is chaotic on $J_f^*$.  In fact, $J_f^*$ is the largest forward invariant  subset of $X$ on which $f$ is chaotic.  Also, $J_f^*$ is the closure of the set of transverse homoclinic points of $f$.
\end{enumerate} 
\end{theorem}

\begin{remark}   \label{r:main-remark}
(1)  If $X$ and $Y$ are Stein manifolds with the density property, then so are $X\times Y$, $X\times\C$, and $X\times\C^*$.  A Stein manifold with the density property is Oka and homogeneous (in the sense that its automorphism group acts transitively on it).  No open Riemann surface has the density property.  It is a long-standing open question whether $\C^{*n}$, $n\geq 2$, has the density property.

Most known examples of Stein manifolds with the density property are captured by the following theorem of Kaliman and Kutzschebauch \cite[Theorem 1.3]{KK2017}.  Let $X$ be a connected affine homogeneous space of a linear algebraic group (for example, a connected linear algebraic group).  If $X$ is not isomorphic to $\C$ or $\C^{*n}$ for some $n\geq 1$, then $X$ has the algebraic density property and therefore also the density property.  For the full list of known Stein manifolds with the density property, see \cite[Section 2.1]{FK2021}.

(2)  A main goal of our work has been to reconcile the \lq\lq attracted versus recurrent\rq\rq\ picture of general dynamics and the \lq\lq calm versus wild\rq\rq\ picture of holomorphic dynamics.  Theorem \ref{t:main-theorem}(d) says that the non-wandering set, a key feature of the former picture, and the Julia set, a key feature of the latter picture, are essentially the same in the generic case.  The analogous result for endomorphisms of Oka-Stein manifolds (which include all Stein manifolds with the density property) is \cite[Theorem 1(d)]{AL2020}.

(3)  A recurrent Fatou component $W$ of an automorphism $f$ of a Stein manifold $X$ is a connected component of $F_f^+$ with a point that has an $\omega$-limit point $p$ in $W$.  Then $W$ is periodic.  If $X$ has the density property and $f$ is generic, then $p$, being a non-wandering point in $F_f^+ = \rne(f)$, is an attracting periodic point \cite[Theorem 2(a)]{AL2020}\footnote{In the proof of \cite[Theorem 2(a)]{AL2020}, we only stated that a non-wandering point of $f$ in $\rne(f)$ is periodic.  Since small perturbations of $f$, obtained using Anders\'en-Lempert theory, remain robustly non-expelling near the point, it must be attracting.} of minimal period $k$, say.  Hence, $W$ is the basin of attraction of one of the points in the cycle of $p$, viewed as an attracting fixed point of $f^k$.  (For automorphisms, as opposed to endomorphisms in general, the basin of attraction of an attracting fixed point is connected.)  Thus,
\[ F_f^+=\bas(f)\sqcup \nrc(f), \]
so $X=F_f^+\sqcup J_f^+$ is partitioned as 
\[ X=\bas(f) \sqcup \nrc(f) \sqcup \overline{W_f^s(q)}, \]
where $q$ is a saddle fixed point of $f$ as in Theorem \ref{t:main-theorem}(b).

(4)  We do not know whether a generic automorphism of a Stein manifold with the density property can have non-recurrent Fatou components.  A related question is open for endomorphisms of an Oka-Stein manifold, see \cite[Section 3]{AL2022}.  If a generic automorphism $f$ has no non-recurrent Fatou components, then by Theorem \ref{t:main-theorem},
\[ C_f=\Omega_f=J_f \cup \att(f) \cup \rep(f). \]

(5)  Another open question is  whether saddle periodic points are dense in the Julia set of a generic automorphism $f$, that is, whether $J^*_f=J_f$.  In Section \ref{sec:closing}, we point out that answering this question in the affirmative is tantamount to establishing a variant of the closing lemma.  By Theorem \ref{t:equality-of-Julia-sets} below, the \lq\lq generic large-cycles closing lemma\rq\rq\ (which is open) is equivalent to any, and hence all, of the following properties for a generic automorphism $f$ of $X$: $J_f^* = J_f$; $f$ is chaotic on $J_f$; periodic points are dense in $J_f$.
\end{remark}

By (4) and (5) above, if the generic large-cycles closing lemma holds, and if a generic automorphism has no non-recurrent Fatou components, then a generic automorphism is chaotic  on the complement of the union of the basins of its attracting and repelling cycles.

We conclude the introduction with a list of notation, mostly the same as established in \cite{AL2022}, an exception being that the right definition of the Julia set of an automorphism is different from the definition that is appropriate in the context of endomorphisms.  In what follows, let $X$ be a complex manifold and let $f$ be an automorphism of $X$.
\begin{itemize}
\item  $\hyp(f)$ is the set of hyperbolic periodic points of $f$.  A periodic point $p$ of $f$ of period $n$ is hyperbolic if the derivative $d_p f^n$ of the $n^\textrm{th}$ iterate $f^n$ at $p$ has no eigenvalue of absolute value $1$.  A periodic point $p$ of minimal period $n$ is transverse if $1$ is not an eigenvalue of $d_p f^n$.
\item  $\att(f)$ is the set of attracting periodic points of $f$.  The periodic point $p$ is attracting if all the eigenvalues of $d_p f^n$ have absolute value less than $1$.
\item  $\rep(f)$ is the set of repelling periodic points of $f$.  The periodic point $p$ is repelling if all the eigenvalues of $d_p f^n$ have absolute value greater than $1$.
\item  $\sad(f)$ is the set of saddle periodic points of $f$.  The periodic point $p$ is a saddle point if it is hyperbolic and some of the eigenvalues of $d_p f^n$ have absolute value less than $1$ and some have absolute value greater than $1$.
\item  $\Omega_f$ is the non-wandering set of $f$, that is, the set of points $p\in X$ such that for every neighbourhood $U$ of $p$, there is $k\geq 1$ such that $U\cap f^k(U)\neq\varnothing$.  Note that $\Omega_f$ is a closed subset of $X$.  Also, $\Omega_{f^{-1}} = \Omega_f$.
\item  $F_f^+$ is the forward Fatou set of $f$, the open set of normality of the forward iterates of $f$.  More explicitly, $F_f^+$ is the set of points in $X$ with a neighbourhood $U$ such that every subsequence of the sequence of forward iterates of $f$ has a subsequence that converges locally uniformly on $U$ to a holomorphic map into $X$ or to the point at infinity.
 The backward Fatou set of $f$ is $F_f^-=F_{f^{-1}}^+$.
\item  $J_f^+=X\setminus F_f^+$ is the forward Julia set of $f$ and $J_f^-=X\setminus F_f^-$ is the backward Julia set of $f$. 
\item  $J_f = J_f^+\cap J_f^- = X\setminus (F_f^+ \cup F_f^-)$ is the Julia set of $f$.  Since $f$ is an automorphism, the sets $F_f^+, F_f^-,J_f^+, J_f^-, J_f$ are all completely invariant, and coincide with the corresponding sets of any iterate $f^n$, $n\neq 0$.
\item  $\rne(f)$ is the open set of points $p\in X$ at which $f$ is robustly non-expelling, meaning that there are neighbourhoods $U$ of $p$ in $X$ and $V$ of $f$ in $\Aut\, X$ and a compact subset $K$ of $X$ such that $g^j(U)\subset K$ for all $g\in V$ and $j\geq 0$.  If $X$ is Stein, by Montel's theorem, $\rne(f)\subset F_f^+$.
\item  $\bas(f)$ is the union of the basins of attraction of the attracting cycles of $f$.
\item $\nrc(f)$ is the union of the non-recurrent Fatou components of $f$.
\end{itemize}

\noindent
{\bf Acknowledgement.}  The authors thank John Erik Forn\ae ss for helpful suggestions concerning homoclinic points.  We also thank an anonymous referee for comments that helped us improve the exposition.

\section{Preliminaries on stable manifolds} 
\label{sec:stable}

\noindent
We briefly recall basic definitions and results on stable manifolds (see for example \cite{PdM1982}).  Let $X$ be a complex manifold of dimension $n$,  let $h$ be an automorphism of $X$, and let $p$ be a saddle fixed point of $h$.  Let $d$ be a distance inducing the topology on $X$.  It is well known that the saddle point moves continuously when $h$ is perturbed.

\begin{lemma}   \label{l:saddlepert}
Let $K\subset X$ be a compact subset containing $p$ in its interior.  Then there exists a neighbourhood $U$ of $p$ and $\epsilon>0$ such that every $\tilde h\in\Aut\,X$ with $d_K(h,\tilde h)<\epsilon$ has a unique fixed point $\eta(\tilde h)$ in $U$, which is a saddle point with the same stable and unstable dimensions as $p$.  Moreover, if $(h_j)$ is a sequence in $\Aut\,X$ such that $d_K(h_j,h)\to 0$, then $\eta(h_j)\to p$.
\end{lemma}

Let $E^s \subset T_pX$ (resp.~$E^u$) be the vector subspace spanned by the generalised eigenvectors corresponding to eigenvalues with absolute value stricly smaller (resp.~larger) than 1.  In suitable holomorphic coordinates centred at $p$, we have $E^s={\rm span}\{e_1,\dots, e_{m}\}$ and  $E^u={\rm span}\{e_{m+1},\dots, e_n\}$.  Denote by $\Delta^n(0,r)$ the polydisc of radius $r>0$ in such coordinates.  The local stable and unstable manifolds of $h$ at $p$ are defined as 
\[ \Gamma^s_h(p,r):=\{z\in \Delta^n(0,r) : h^j(z)\in \Delta^n(0,r) \textrm{ for all } j\geq 0\} \]
and
\[ \Gamma^u_h(p,r):=\{z\in \Delta^n(0,r) : h^{-j}(z)\in \Delta^n(0,r)\textrm{ for all } j\geq 0\}. \]
If $r$ is sufficiently small, then $\Gamma^s_h(p,r)$ is contained in the stable manifold
\[ W^s_h(p):=\{z\in X: h^j(z)\to p \textrm{ as } j\to\infty\}, \] 
and $\Gamma^u_h(p,r)$ is contained in the unstable manifold
 \[ W^u_h(p):=\{z\in X: h^{-j}(z)\to p \textrm{ as } j\to\infty\}. \] 
Moreover, if $r$ is sufficiently small, the local stable manifold is the   graph  of a holomorphic map $\varphi: \Delta^m(0,r)\to \Delta^{n-m}(0,r)$ with $d_0\varphi=0$.  Clearly, every orbit converging to $p$ is eventually contained in $\Gamma^s_h(p,r)$, so 
\[ \bigcup_{j\geq 0}h^{-j}(\Gamma^s_h(p,r)) = W^s_h(p),\] 
which shows that the stable manifold is an  immersed complex submanifold biholomorphic to  $\C^m$ and  tangent to $E^s$ at $p$.  Analogous considerations hold for the unstable manifold.  Next we consider how the stable manifold changes under small perturbations of the automorphism $h$.

\begin{lemma}   \label{l:delicate}
If $r>0$ is small enough, the following holds.  Let $K\subset X$ be compact and contain $\overline{\Delta^n(0,r)}$ in its interior.  Let $(h_j)$ be a sequence in $\Aut\,X$ with $d_K(h_j,h)\to 0$.  For $j$ large enough, let $p_j:=\eta(h_j)$ be the saddle fixed point given by Lemma \ref{l:saddlepert}.  Then there is a holomorphic map $\varphi_j: \Delta^m(0,r)\to \Delta^{n-m}(0,r)$, whose graph is contained in the stable manifold $W^s_{h_j}(p_j)$, such that $\varphi_j$ converges to $\varphi$ uniformly on $\Delta^m(0,r)$ as $j\to\infty$.
\end{lemma}
 
Take $q\in W^s_h(p), q\neq p$.  Choose holomorphic coordinates centred at $q$ such that $T_qW^s_h(p)={\rm span}\{e_1,\ldots, e_m\}$.  By Lemma \ref{l:delicate}, pulling back by $h$, if $r$ is sufficiently small, the stable manifold $W^s_h(p)$ contains the graph of a holomorphic map $\psi: \Delta^m(0,r)\to \Delta^{n-m}(0,r)$ with $d_0\psi=0$.  The following result now follows from the lemma.

\begin{lemma}   \label{l:delicate2}
Let $K\subset X$ be compact and contain $\{p\}\cup \{h^j(q): j\geq 0\}$ in its interior.  If $r>0$ is small enough, the following holds.  Let $(h_j)$ be a sequence in $\Aut\,X$ with $d_K(h_j,h)\to 0$.  For $j$ large enough, let $p_j:=\eta(h_j)$ be the saddle fixed point given by Lemma \ref{l:saddlepert}. Then there is a holomorphic map $\psi_j: \Delta^m(0,r)\to \Delta^{n-m}(0,r)$, whose graph is contained in the stable manifold  $W^s_{h_j}(p_j)$, such that $\psi_j$ converges to $\psi$ uniformly on $\Delta^m(0,r)$ as $j\to\infty$.
\end{lemma}

Next we recall the following well-known result.  For want of a reference we supply a proof.

\begin{lemma}  \label{l:stable-in-julia}
Let $f$ be an automorphism of a complex manifold $X$.  Let $p$ be a saddle fixed point of $f$.  Then $W^s_f(p)\subset J_f^+$ and $W^u_f(p)\subset J_f^-$.
\end{lemma}

\begin{proof}
Let $\|\cdot\|$ be a hermitian metric on $X$.  Suppose that there is $x\in W^s_f(p)\cap F_f^+$.  By normality, for any tangent vector $v\in T_x X$, the sequence $(\|d_xf^n(v)\|)$ is bounded.

Let $\mathscr{C}_p$ be a homogeneous cone in $T_pX$ such that the tangent space to the unstable manifold $W_f^u(p)$ is contained in the interior of $\mathscr{C}_p$ and the tangent space to the stable manifold $W_f^s(p)$ is contained in the interior of $T_p X\setminus \mathscr{C}_p$.  For $k\geq 0$ large enough, $d_pf^k(\mathscr{C}_p)\subset {\rm int}(\mathscr{C}_p)$ and $\|d_pf^k(v)\|> 2\|v\|$ for all $v\in \mathscr{C}_p\setminus\{0\}$.  The cone $\mathscr{C}_p$ can be extended to a cone field $\mathscr{C}_\bullet$ in a neighbourhood $U$ of $p$ such that
\[ d_zf^k(\mathscr{C}_z)\subset {\rm int}(\mathscr{C}_{f(z)}) \quad\textrm{and}\quad \|d_zf^k(v)\|> 2\|v\| \textrm{ for all } v\in \mathscr{C}_z\setminus\{0\} \]
for all $z\in U\cap f^{-1}(U)$.

Since $x\in  W^s_f(p)$, there exists an integer $m$ such that $f^n(x)\in U$ for all $n\geq m$.  Choose a nonzero $v\in T_xX$ such that $d_xf^m(v)\in \mathscr{C}_{f^m(x)}\setminus\{0\}$.  Then
\[\|d_xf^{m+kj}(v)\|> 2^k\|d_xf^m(v)\|\]
for all $j\geq 1$, which contradicts the boundedness of the sequence $(\|d_xf^n(v)\|)$.
\end{proof}

\section{The Kupka-Smale theorem}
\label{sec:Kupka-Smale}

\noindent
In this section we generalise the Kupka-Smale's theorem to our setting.  Here is the first half.

\begin{theorem}   \label{t:Kupka-Smale-1}
For a generic automorphism of a Stein manifold $X$ with the density property, every periodic point is hyperbolic.
\end{theorem}

\begin{proof}
The proof is so similar to the proof in \cite{AL2022} of the analogous result for endomorphisms of an Oka-Stein manifold \cite[Theorem 1(a)]{AL2022} that we will only sketch it.

Given a compact subset $K$ of $X$ and an integer $m\geq 1$, let $\mathscr T(m,K)\subset\Aut\,X$ be the open set of automorphisms such that every periodic point in $K$ of $X$ of minimal period at most $m$ is transverse.  Let $\mathscr H(m,K)\subset\Aut\,X$ be the open set of automorphisms of $X$ such that every period point in $K$ of minimal period at most $m$ is hyperbolic.  

We prove, in three steps as in \cite{AL2022}, that for all $K$ and $m$, $\mathscr H(m,K)$ is dense in $\Aut\, X$.  We then exhaust $X$ by compact sets $K_m$ and conclude that every periodic point of an automorphism in the residual subset $\bigcap\limits_{m\geq 1}\mathscr H(m,K_m)$ of $\Aut\, X$ is hyperbolic.

\smallskip\noindent
{\it Step 1.}  {\it For every compact $K\subset X$, $\mathscr T(1,K)$ is dense in $\Aut\, X$.}  This the analogue of \cite[Proposition 1]{AL2022}.  Let $f_0\in\Aut\, X$.  We will produce a perturbation of $f_0$, that is, a continuous map $f:P\to \Aut\, X$, where the parameter space $P$ is a neighbourhood of the origin $0$ in some $\C^k$, with $f(0)=f_0$, such that there are arbitrarily small $t\in P$ with $f_t=f(t)\in \mathscr T(1,K)$.  Equivalently, for arbitrarily small $t\in P$, the map $X\to X\times X$, $x\mapsto (x,f_t(x))$, is transverse to the diagonal $\Delta\subset X\times X$ on $K$.  By the parametric transversality theorem, this holds if the associated map $F:X\times P\to X\times X$, $F(x,t)=(x,f_t(x))$, is holomorphic (or merely $C^1$) and transverse to $\Delta$ on $K\times P$.

Take finitely many complete holomorphic vector fields $v_1,\ldots,v_r$ on $X$ that span each tangent space of $X$ on a neighbourhood $U$ of $f_0(K)$.  Let $\psi_\bullet^j$ be the flow of $v_j$ and let $f_t=\psi_{t_r}^r\circ \cdots\circ\psi_{t_1}^1\circ f_0\in \Aut\,X$, $t\in\C^r$.  Then $\dfrac\partial{\partial t_j}f(x,t)\vert_{(x,0)} = v_j(f_0(x))$ for all $x\in X$, so the associated map $F$ is transverse to the diagonal on $f_0^{-1}(U)\times\{0\}$ and hence on $K\times P$ for a sufficiently small neighbourhood $P$ of $0$ in $\C^k$.

\smallskip\noindent
{\it Step 2.}  {\it For every compact $K\subset X$ and $m\geq 1$, $\mathscr H(m,K)$ is dense in $\mathscr T(m,K)$.}  This is the analogue of \cite[Proposition 2]{AL2022} and the proof is verbatim the same, except with $\End\,X$ replaced by $\Aut\,X$.  The proof of \cite[Proposition 2]{AL2022} refers to the endomorphism case of \cite[Lemma 1]{AL2020}; here we need the automorphism case of the lemma.

\smallskip\noindent
{\it Step 3.}  {\it For every compact $K\subset X$ and $m\geq 2$, $\mathscr T(m,K) \cap \mathscr H(m-1,K)$ is dense in $\mathscr H(m-1,K)$.}  This is the analogue of \cite[Proposition 3]{AL2022} and the proof is the same, except the perturbation $f:X\times (\C^r)^\ell\to X$ of the given automorphism $f_0\in \mathscr H(m-1,K)$ needs to be defined differently.  As in the proof of \cite[Proposition 3]{AL2022}, we define $f_t=f(\cdot, t_1,\ldots,t_\ell) = \phi_{t_1}^{V_1}\circ\cdots\circ \phi_{t_\ell}^{V_\ell}\circ f_0$, but now, $\phi_\bullet^{V_1},\ldots, \phi_\bullet^{V_\ell}$ are holomorphic families of automorphisms, parametrised by $\C^r$, of the kind $\phi_\bullet^V$ described as follows.

In the notation of the proof of \cite[Proposition 3]{AL2022}, were it possible to find finitely many complete holomorphic vector fields $v_1,\ldots,v_r$ on $X$ that span each tangent space of $X$ on a neighbourhood of $\overline V$ and are as close to zero as we wish on $L_1\cup \cdots\cup L_{m-1}$, then we would take $\phi_t^V=\psi_{t_r}^r\circ \cdots\circ\psi_{t_1}^1$, with $\psi_\bullet^j$ being the flow of $v_j$ (as in Step 1).

Without the completeness requirement, the fields $v_1,\ldots,v_r$ can be chosen as above.  Using the density property, then, we approximate $v_j$ uniformly on $\overline V\cup L_1\cup \cdots\cup L_{m-1}$ by a sum $v=u_1+\cdots+u_s$ of complete holomorphic vector fields $u_1,\ldots,u_s$ on $X$ with flows $\eta_\bullet^1,\ldots, \eta_\bullet^s$.  Then $\eta_\tau=\eta_\tau^s\circ\cdots \circ\eta_\tau^1$ is an algorithm for $v$, meaning that $\eta_0=\id_X$ and $\dfrac\partial{\partial\tau}\bigg\vert_{\tau=0}\eta_\tau(x)=v(x)$ for all $x\in X$, and we take $\psi_{t_j}^j=\eta_{t_j}$ and $\phi_t^V=\psi_{t_r}^r\circ \cdots\circ\psi_{t_1}^1$.
\end{proof}

The following result can be used to slightly simplify the proof of Theorem \ref{t:Kupka-Smale-1}.  We include it here for possible use elsewhere.  Under the assumption that the manifold $X$ is Stein, the result is implicit in the proof of \cite[Theorem 4]{KK2008}.

\begin{proposition}
Let $X$ be a connected complex manifold with a set of complete holomorphic vector fields that span each tangent space of $X$.  Then there is a finite such set.
\end{proposition}

\begin{proof}
Take finitely many complete vector fields $v_1,\ldots,v_m$ that span $T_p X$ for some $p\in X$ and hence for every $p$ outside a proper closed analytic subset $A$ of $X$.  We want to produce finitely many fields that span $T_q X$ for a point $q$ in each irreducible component of $A$ and then induct on dimension.  Let the countable set $D$ consist of one point from each component of $A$.  

By \cite[Corollary 8]{Winkelmann2021}, there is an automorphism $f$ of $X$ with $f(x)\notin A$ for every $x\in D$.  (The corollary applies because by hypothesis, the orbits of $\Aut\,X$ are open, so there is only one orbit.  Also, as for an arbitrary complex manifold, $\Aut\,X$ is completely metrisable, and hence Baire, and separable.)  Then the inverse of $f$ takes a point in $X\setminus A$ into each component of $A$, so the fields $f_*^{-1}v_1,\ldots, f_*^{-1}v_m$ span $T_q X$ for a point $q$ in each component of $A$.
\end{proof}

Here is the second half of the Kupka-Smale theorem.  We will describe how Buzzard's proof for $\C^n$, $n\geq 2$, can be modified so as to work in our setting.

\begin{theorem}   \label{t:Kupka-Smale-2}
For a generic automorphism of a Stein manifold $X$ with the density property, every homoclinic or heteroclinic point is transverse.
\end{theorem}

\begin{proof}
In Buzzard's notation, let $p_1$ and $p_2$ be saddle periodic points of an automorphism $F$ of $X$ and $q_0 \in W^s_F(p_1)\cap W^u_F(p_2)$ be a homoclinic or heteroclinic point.  We need suitable replacements for his families $\Psi_k$, $k\geq 1$, of automorphisms and family $\Psi$ of diffeomorphisms of $X$, defined in \cite[page 501]{Buzzard1998}.

The total $F$-orbit of $q_0$ accumulates only on the union of the cycles of $p_1$ and $p_2$, which is a finite set.   Let $K$ be the compact closure of the orbit.  Let $U$ be a Runge neighbourhood of $K$, consisting of finitely many coordinate balls with mutually disjoint closures, such that $q_0$ is the only point of $K$ in the ball $U_0$ that contains it.  Let $U_1 = U\setminus U_0$.  Let $V_0\subset U_0$ be a smaller ball containing $q_0$ and let $B$ be a closed ball in $\C^n$, $n=\dim X$, centred at the origin, so that $\overline V_0+z\subset U_0$ for all $z\in B$ with respect to the coordinates in $U_0$.  Write $V=V_0\cup U_1$.

Define $\Phi:V\times B\to X$ by the formula $(x,z)\mapsto x+z$ on $V_0\times B$ and by the formula $(x,z)\mapsto x$ on $U_1\times B$.  Extend $\Phi$ to a smooth family $\Psi:X\times B\to X$ of diffeomorphisms.  Use the parametric Anders\'en-Lempert theorem \cite[Theorem 1.1]{Forstneric1994}\footnote{The theorem is stated for $\C^n$, $n\geq 2$, but holds more generally for Stein manifolds with the density property.  The Runge domains in the theorem should then be taken to be Stein (as they are here), but they need not be connected.} to approximate $\Phi$ locally uniformly on $V\times B$ by a sequence of smooth families $\Psi_k:X\times B\to X$ of automorphisms of $X$.  

These families have the properties needed for Buzzard's proof.  The remainder of the proof consists of local arguments and transversality and genericity arguments that straightforwardly extend to our more general setting.
\end{proof}

\section{Julia set and non-wandering set}
\label{sec:Julia}

\noindent
Our next result is the technical heart of the paper.

\begin{theorem}  \label{t:stable-unstable}
A generic automorphism $f$ of a Stein manifold $X$ with the density property has a saddle fixed point $q$ such that
\[ X\setminus\rne(f) = \overline{W_f^s(q)} \qquad \textrm{and} \qquad X\setminus\rne(f^{-1}) = \overline{W_f^u(q)}. \]
\end{theorem}

Before proving the theorem we need a  lemma.

\begin{lemma}   \label{l:create-point}
Let $X$ be a Stein manifold with the density property and $f$ be an automorphism of $X$.  Let $K\subset X$ be a compact subset and let $\epsilon>0$.  Then there is an automorphism $g$ of $X$ such that $d_K(g,f)<\epsilon$, with a fixed point in $X\setminus K$, which can be chosen to be attracting, saddle, or repelling. 
\end{lemma}

\begin{proof}
After enlarging $K$, we may assume that it is holomorphically convex.  Let $p\notin f(K)\cup K$.  By \cite[Theorem 2]{Varolin2000}, there is $h\in\Aut\,X$ such that $d_{f(K)}(h,\id)<\epsilon$, $h(f(p))=p$, and $d_{f(p)}h\circ d_pf$ is a saddle or attracting or repelling.  Now take $g=h\circ f$.
\end{proof}

\begin{remark}   \label{r:infinitely-many}
It immediately follows from Lemma \ref{l:create-point} that a generic automorphism of a Stein manifold $X$ with the density property has infinitely many attracting fixed points, infinitely many repelling fixed points, and infinitely many saddle fixed points. Indeed, given a compact $K\subset X$, the open set of automorphisms of $X$ with an attracting fixed point outside $K$ is dense in $\Aut\,X$.  The same argument works for repelling and saddle fixed points.  Note that this implies that for a generic automorphism $f$ of $X$, the sets $J^+_f, J^-_f, J_f, F^+_f, F^-_f, \Omega_f$ are not relatively compact.
\end{remark}

\begin{proof}[Proof of Theorem \ref{t:stable-unstable}]
 By Remark \ref{r:infinitely-many}, there is a countable dense subset $\{\tilde \psi_1, \tilde \psi_2, \ldots\}$ in $\Aut\,X$ such that each $\tilde \psi_j$ admits a saddle fixed point $\eta(\tilde \psi_j)$.  After passing to a suitable subsequence of $(\tilde \psi_j)$, which we call $(\psi_j)$, we can construct by induction compact subsets $H_j\subset X$ and numbers $\gamma_j>0$ such that 
\begin{enumerate}
 \item  $\eta(\psi_j)\in H_j^\circ$,
 \item the open set $B_j:=\{h\in \Aut\,X: d_{H_j}(\psi_j,h)<\gamma_j\}$ is disjoint from $B_1,\ldots,B_{j-1}$, 
 \item $A:=\bigcup\limits_{j\geq 1}B_j$ is  dense  in $\Aut\,X$ since it contains  $\{\tilde \psi_1, \tilde \psi_2, \ldots\}$,
 \item every automorphism $h\in B_j$ has a unique saddle fixed point $\eta(h)$ in the neighbourhood  of $\eta(\psi_j)$ provided by Lemma \ref{l:saddlepert}, which we can assume to be contained in $H_j^\circ$.
 \end{enumerate}

Let $\{U_n:n\geq 1\}$ be a countable basis for the topology of $X$ consisting of relatively compact open sets, and define the  sets
\[ S_n:=\{f\in A: W^s_f(\eta(f))\cap U_n\neq \varnothing\}\] 
and 
\[ T_n:=\{f\in A: W^u_f(\eta(f))\cap U_n\neq \varnothing\},\]
which are open by Lemma \ref{l:delicate2}.  Hence,
\[ G:= \bigcap_n A\setminus \partial S_n\,\cap\, \bigcap_n A\setminus \partial T_n \]
is residual in $\Aut\,X$.  We claim that for all $f\in G$, the stable manifold $W^s_f(\eta(f))$ is dense in $X\setminus \rne(f)$ and the unstable manifold $W^u_f(\eta(f))$ is dense in $X\setminus \rne(f^{-1})$.  By Lemma \ref{l:stable-in-julia}, this proves the theorem.

Fix $f\in G$ and $U_n$ intersecting $X\setminus \rne(f)$.  We will show that $f\in S_n$.  By the definition of $G$, this follows if we show that $f\in \overline S_n$.  Let $K\subset X$ be compact and let $\epsilon >0$.  We will show that there is $g\in S_n$ such that $d_K(f,g)<\epsilon$.  An analogous argument shows that $f\in T_n$.

Find $j\geq 1$ such that $f\in B_j$.  Let $H$ be a holomorphically convex compact set containing $K\cup H_j$.  By definition of $\rne(f)$, there is $x\in U_n$ and an automorphism $\tilde f\in B_j$ with $d_H(f,\tilde f)< \epsilon/2$, such that the orbit $({\tilde f}^n(x))$ is not contained in $H$.  Let $n_0\geq 0$ be the smallest nonnegative integer such that ${\tilde f}^{n_0}(x)\in X\setminus H$.  The stable manifold $W^s_{\tilde f} (\eta(\tilde f))$ is not contained in $H$, so we may choose $y\in W^s_{\tilde f} (\eta(\tilde f))\setminus H$ such that $\tilde f^n(y)\in H$ for all $n\geq 1$.

Let $W$ be a Runge\footnote{We take a Runge open set to be Stein by definition.} neighbourhood of $H$ containing neither $\tilde f^{n_0}(x)$ nor $y$.  Let $V$ be a neighbourhood of $\tilde f^{n_0}(x)$ and let $\varphi : [0,1]\times V\to X$ be a $C^1$ isotopy such that for all $t\in [0,1]$,
\begin{enumerate}
\item $\varphi_t : V\to X$ is holomorphic and injective,
\item $\varphi_t(V)$ is disjoint from $W$, 
\item $W\cup \varphi_t(V)$ is Runge,
\item $\varphi_0$ is the inclusion of $V$ into $X$, 
\item $\varphi_1(\tilde f^{n_0}(x))=y$.
\end{enumerate}

Find $m_0\geq 1$ such that $\tilde f^{m_0}(y)$ belongs to the local stable manifold $\Gamma^s_{\tilde f}(\eta(\tilde f),r)$, where $r$ is given by Lemma \ref{l:delicate}.  There are relatively compact neighbourhoods $U\subset U_n$ of $x$ and $Z$ of $\tilde f^{m_0}(y)$ such that ${\tilde f}^{n_0}(U)\Subset V$ and $Z\Subset \tilde f^{m_0}\circ  \varphi_1\circ {\tilde f}^{n_0}(U)$.

By the Anders\'en-Lempert theorem \cite[Theorem 4.10.5]{Forstneric2017}, there is a sequence $(\Phi_j)$ in $\Aut\,X$ such that $\Phi_j\to \id$ on $W$ and $\Phi_j\to \varphi_1$ on $V$, uniformly on compact subsets.  For $j$ large enough, the automorphism $g=\tilde f\circ \Phi_j$ satisfies the following conditions:
\begin{enumerate}
\item $g\in B_j$.
\item $d_{H}(\tilde f,g)<\epsilon/2$.
\item $Z\Subset g^{m_0+n_0}(U)$.
\item  $ W^s_{g}(\eta(g))\cap Z\neq \varnothing$.
\end{enumerate}
Hence $U_n$ intersects the stable manifold $W^s_{g}(\eta(g))$.
\end{proof}

This result has several corollaries.  The first follows immediately from Theorem \ref{t:stable-unstable} and Lemma \ref{l:stable-in-julia}.

\begin{corollary}  \label{c:stable-unstable-Julia}
For a generic automorphism $f$ of a Stein manifold $X$ with the density property,
\[ F_f^+ = \rne(f), \qquad J_f = X \setminus (\rne(f) \cup \rne(f^{-1})), \]
and $f$ has a saddle fixed point $q$ such that
\[ J_f = \overline{W^s_f(q)}\cap \overline{W^u_f(q)}. \]
\end{corollary}

\begin{remark}
The forward Fatou set $F_f^+$ is defined using the one-point compactification of $X$.  Different compactifications can be used, resulting in a forward Fatou set that is in general strictly smaller than the set  $F_f^+$ defined above, see for example \cite[Example 2.6]{ABFP}.  Surprisingly, for a generic $f$ this ambiguity disappears.  Indeed, Corollary \ref{c:stable-unstable-Julia} shows that for a generic $f$ the definition of the forward Fatou set is independent of the choice of compactification, since every forward Fatou set contains $\rne(f)$.

The same holds for the definitions of $F_f^-$, $J_f^+$, $J_f^-$, and $J_f$.
\end{remark} 

\begin{corollary}\label{c:NW}
For a generic automorphism $f$ of a Stein manifold $X$ with the density property the following holds. Let $k\geq 1$.  If $U$ is a neighbourhood of a point in $J_f^+$, $W$ is a neighbourhood of a point in  $J_f^-$, and  $V$ is a neighbourhood of the fixed saddle point $q$ given by Corollary \ref{c:stable-unstable-Julia}, then there is a point in $U$ whose forward $f^k$-orbit enters $V$ and subsequently enters $W$. In particular  there is $n>0$ such that 
\[ f^{nk}(U)\cap W\neq \varnothing.\]
\end{corollary}

\begin{proof}
By Theorem \ref{t:stable-unstable}, $J_f^+ =\overline{W_f^s(q)}=\overline{W_{f^k}^s(q)}$ and $J_f^- =\overline{W_f^u(q)}=\overline{W_{f^k}^u(q)}$.  Let $D_s$ be a  polydisc transverse to $W_{f^k}^s(q)$ contained in $U$, and let $D_u$ be a polydisc transverse to $W_{f^k}^u(q)$ contained in $W$.  By the lambda lemma applied to the map $f^k$, there is a point in  $D_s\subset U$ whose $f^k$-orbit enters $V$ and subsequently lands in  $D_u\subset W$.
\end{proof}

\begin{corollary}   \label{t:forward-Julia}
For a generic automorphism $f$ of a Stein manifold $X$ with the density property, the following hold.
\begin{enumerate}
\item[(a)]  The forward Julia set $J_f^+$ is connected.
\item[(b)]  $J_f^+$ is the boundary of each connected component of the basin of attraction of every attracting cycle of $f$.
\item[(c)]  $J_f^+$ and $J_f$  have empty interior.
\item[(d)]  If $U$ is a neighbourhood of a point in $J_f^+$, then $U\setminus J_f^+$ has infinitely many connected components.  Hence, $J_f^+$ is not an embedded topological manifold at any of its points.
\item[(e)]  Every connected component of $F_f^+$ is Stein.
\end{enumerate}
\end{corollary}

\begin{proof}
(a)  By Theorem \ref{t:stable-unstable}, $f$ has a saddle fixed point $q$ such that $J_f^+$ is the closure of the stable manifold $W_f^s(q)$, which is connected.

(b)  Let $B$ be the basin of attraction of an attracting cycle of $f$.  (By Remark \ref{r:infinitely-many}, there are infinitely many such basins.)  Let $k$ denote the length of the cycle.  Let $B_0$ be a connected component of $B$.  Then $B_0$ is the basin of attraction with respect to $f^k$ of a point $p$ in the cycle.  We have $p\in J_f^-$.  Let $U$ be a neighbourhood of a point in $J_f^+$.  By Corollary \ref{c:NW}, $f^{jk}(U)$ intersects $B_0$ for some $j\geq 0$, so $U$ intersects $B_0$.  This shows that $J_f^+\subset \partial B_0$.  The opposite inclusion is evident.

(c)  Being the boundary of an open set by (b), $J_f^+$ has empty interior.  It follows that $J_f$ has  empty interior too.

(d)  By Remark \ref{r:infinitely-many}, $f$ has infinitely many attracting fixed points.  Each basin of attraction of $f$ is a connected component of $F_f^+ = X\setminus J_f^+$.

(e)  Let $W$ be a connected component of $F_f^+=\rne(f)$.  The characterisation of pseudoconvex domains in $\C^n$ in terms of Hartogs figures extends to Stein manifolds.  Let $H$ be a Hartogs figure in a polydisc $P$ of dimension $\dim X$ and let $\phi:P\to X$ be a holomorphic map with $\phi(H)\subset W$.  The maximum principle shows that $\phi(P)\subset \rne(f)$, so by connectedness, $\phi(P)\subset W$.
\end{proof}

\begin{theorem}  \label{t:non-wand-equals-Julia}
For a generic automorphism $f$ of a Stein manifold $X$ with the density property,
\[ \Omega_f = J_f \cup \att(f) \cup \rep(f). \]
\end{theorem}

\begin{proof}
Corollary \ref{c:NW} yields $J_f\subset \Omega_f$.  Conversely, let $x$ be a non-wandering point of an automorphism $f$ of $X$ outside $J_f$.  For generic $f$, by Corollary \ref{c:stable-unstable-Julia}, $x$ lies in $\rne(f)$ or $\rne(f^{-1})$.  If $x\in\rne(f)$, then $x\in\att(f)$, and if $x\in\rne(f^{-1})$, then $x\in\att(f^{-1})=\rep(f)$ by \cite[Step 1 of the proof of Theorem 2]{AL2020}.
\end{proof}

\begin{corollary}  \label{c:empty-interior}
For a generic automorphism $f$ of a Stein manifold $X$ with the density property, the  non-wandering set $\Omega_f$ has empty interior.
\end{corollary}

\begin{proof}
By Theorem \ref{t:non-wand-equals-Julia}, $\Omega_f = J_f \cup \att(f) \cup \rep(f)$.  By Corollary \ref{t:forward-Julia}, $J_f$ has empty interior.  Finally, $\att(f)\cup \rep(f)$ is discrete.
\end{proof}

\section{Chain-recurrent set} 
\label{sec:Conley}

\noindent
Conley's general theory of chain-recurrence for an endomorphism of a topological space satisfying certain mild hypotheses  was introduced in \cite{Conley1978}.  We consider the  stronger notion of chain-recurrence introduced by Hurley in \cite{Hurley1992}, which is better suited to non-compact spaces since it does not depend on the choice of a metric (see also \cite[Section 4]{AL2022}).  Let $X$ be a locally compact second countable metric space and let $f:X\to X$ be continuous.  Choose a metric $d$ on $X$ compatible with the topology of $X$.  Let $\epsilon: X\to(0,\infty)$ be continuous.  A finite sequence $x_0, x_1,\ldots, x_n$, $n\geq 1$, of points in $X$ is an $\epsilon$-chain or $\epsilon$-pseudo-orbit of length $n$ if $d(f(x_j), x_{j+1})<\epsilon (f(x_j))$ for $j=0,\ldots,n-1$.  A point $p$ in $X$ is chain-recurrent for $f$ if for every function $\epsilon$, there is an $\epsilon$-chain that begins and ends at $p$.  We denote by $C_f$ the set of chain-recurrent points of $f$.  
An equivalence relation is defined on $C_f$ by declaring points $p$ and $q$ equivalent if for every continuous $\epsilon:X\to (0,\infty)$, there is an $\epsilon$-chain from $p$ to $q$ and an $\epsilon$-chain from $q$ to $p$.  The equivalence classes are called chain-recurrence classes.  Note that $\Omega_f\subset C_f$.  Also, if $f$ is a homeomorphism, then $C_f=C_{f^{-1}}$.

In \cite[Section 3]{AL2022}, we introduced the notion of a pre-recurrent Fatou component of an endomorphism.  For automorphisms, it coincides with the notion of a recurrent Fatou component.  By \cite[Section 4]{AL2022}, the chain-recurrent set of a generic endomorphism $f$ of an Oka-Stein manifold consists of the chain-recurrence class $J_f^+\cup\npr(f)$ along with the attracting cycles of $f$.  Here, $\npr(f)$ is the union of the non-pre-recurrent Fatou components of $f$.  It is an open question whether $\npr(f)$ is empty for generic $f$.  Our next theorem is the corresponding result for automorphisms.  
The {\it proper basin} of an attracting cycle is the basin of attraction of the cycle with the cycle itself removed.
\begin{theorem}   \label{t:Conley}
For a generic automorphism $f$ of a Stein manifold $X$ with the density property, the following hold.
\begin{enumerate}
\item[(a)]  A point in $X$ is chain-recurrent if and only if it does not lie in the proper basin of an attracting or repelling cycle of $f$.
\item[(b)]  The chain-recurrence classes are the following.
\begin{itemize}
\item  The complement of the union of the basins of the attracting and repelling cycles of $f$.  This class may also be described as
\[ \big(J_f^+ \cup \nrc(f)\big) \cap \big(J_f^- \cup \nrc(f^{-1})\big). \]
\item  Each attracting cycle and each repelling cycle is a chain-recurrence class.
\end{itemize}
\end{enumerate}
\end{theorem}
\begin{proof}
It is clear that a point in the proper basin of an attracting or repelling cycle is not chain-recurrent.
One cannot escape from an attracting cycle along an $\epsilon$-pseudo-orbit for $f$ if $\epsilon$ is small enough.  Likewise, one cannot reach a repelling cycle along an $\epsilon$-pseudo-orbit for $f$ if $\epsilon$ is small enough.  Thus, each attracting cycle and each repelling cycle is a chain-recurrence class of its own.

By Remark \ref{r:main-remark}(3),
\[ X=J_f^+\sqcup \nrc(f)\sqcup \bas(f)=J_f^-\sqcup  \nrc(f^{-1})\sqcup \bas(f^{-1}),\]
so
\[ X\setminus (\bas(f)\cup \bas(f^{-1}))= \big(J_f^+ \cup \nrc(f)\big) \cap \big(J_f^- \cup \nrc(f^{-1})\big). \]

It follows immediately from Corollary \ref{c:NW} that $J_f$ lies in a single chain-recurrence class.  

Now take $p\in \nrc(f) \cap J_f^-$.  The forward orbit of $p$ is relatively compact because $\nrc(f)\subset F_f^+=\rne(f)$, so $p$ has an $\omega$-limit point $s\in J_f^+$.  By Corollary \ref{c:NW}, arbitrarily close to $s$ is a point whose forward orbit comes arbitrarily close to a saddle fixed point and subsequently comes arbitrarily close to $p$.  This shows that $p$ is chain-recurrent and lies in the same chain-recurrence class as $J_f$.  The case of $p\in \nrc(f^{-1}) \cap J_f^+$ is analogous.

Finally, suppose that $p\in \nrc(f)\cap \nrc(f^{-1})$.  Then $p$ has an $\alpha$-limit point $r\in J_f^-$ and an $\omega$-limit point $s\in J_f^+$.  By Corollary \ref{c:NW}, arbitrarily close to $s$ is a point whose forward orbit comes arbitrarily close to a saddle fixed point and subsequently comes arbitrarily close to $r$.  Thus, $p$ is chain-recurrent and lies in the same chain-recurrence class as $J_f$.
\end{proof}

\begin{remark}
The first fundamental theorem of Conley, as adapted to the non-compact case by Hurley, states that the chain-recurrent points are precisely those that lie in no proper  basin in the sense of Conley.  This abstract notion of a basin (which we refer to as a Conley basin for clarity) was introduced in  \cite{Conley1978} and \cite{Hurley1992} (see also \cite[Section 4]{AL2022}); by a proper Conley basin we mean a Conley basin with its attractor removed.  Theorem \ref{t:Conley} shows that for a generic automorphism $f$ of a Stein manifold with the density property, the proper Conley basins  have the same union as the proper basins of the attracting and repelling cycles of $f$.
\end{remark}

The following result is now nearly evident.
\begin{corollary}  \label{c:equality-of-non-wandering-and-chain-recurrent}
For a generic automorphism $f$ of a Stein manifold with the density property, the following are equivalent.
\begin{enumerate}
\item[(i)]  $C_f=\Omega_f$.
\item[(ii)]  $J_f$ is a chain-recurrence class.
\item[(iii)]  $F_f^+\cup F_f^- = \bas(f)\cup \bas(f^{-1})$.
\item[(iv)]  $\rne(f)\cup\rne(f^{-1})= \bas(f)\cup \bas(f^{-1})$.
\item[(v)]  $\nrc(f)\subset\bas(f^{-1})$ and $\nrc(f^{-1}) \subset \bas(f)$.
\item[(vi)]  $\nrc(f^{-1})\cup (J_f^-\setminus J_f^+) \subset \bas(f)$.  
\end{enumerate}
\end{corollary}

\begin{proof}
To see that (v) and (vi) are equivalent, note that given $\nrc(f^{-1}) \subset \bas(f)$, $\nrc(f)\subset \bas(f^{-1})$ is equivalent to $J_f^- \subset J_f^+\cup \bas(f)$, that is, $J_f^-\setminus J_f^+\subset \bas(f)$.
\end{proof}

\section{Chaotic Julia set}
\label{sec:chaotic-Julia}

\noindent
We define the {\it chaotic Julia set} $J_f^*$ of an automorphism $f$ of a Stein manifold $X$ to be the closure of the set $\sad(f)$ of saddle periodic points of $f$ and the {\it non-chaotic Julia set} $M_f=J_f\setminus J_f^*$ to be its complement in the Julia set $J_f$.  Evidently, $J_f^*$ is completely invariant.  Also, $J_{f^m}^*=J_f^*$ for all integers $m\neq 0$ because the saddle periodic points of $f$ and $f^m$ are the same.  In Theorem \ref{t:chaotic-and-non-chaotic-Julia-sets} below we describe the main properties of the chaotic and non-chaotic Julia sets.  First we need to establish some properties of the set $\tam(f)$ of {\it tame points} of $f$.  This notion emerged in \cite{AL2020} from our efforts to understand the proof of the closing lemma \cite[Theorem~5.1]{FS1997} for automorphisms of $\C^n$.  For a compact $K\subset X$, we define the following closed subsets of $X\times\Aut\, X$:
\[ T_K^+ = \{(x,g)\in X\times\Aut\, X: g^j(x)\in K \textrm{ for all } j\geq 0\}, \]
\[ T_K^- = \{(x,g)\in X\times\Aut\, X: g^j(x)\in K \textrm{ for all } j\leq 0\}. \]
We say that $p\in X$ is tame for $f\in\Aut\, X$ if whenever $(p,f)\in (T_K^+\cup T_K^-)^\circ$ for a compact $K \subset X$, we have $(p,f)\in \overset\circ {T_L^+} \cup \overset\circ {T_L^-}$ for some, possibly larger, compact $L \subset X$.  In particular, $p$ is tame for $f$ if $(p,f)\not\in (T_K^+\cup T_K^-)^\circ$ for all compact $K \subset X$. For example, saddle periodic points are tame.  Note that $p$ is tame for $f$ if and only if $p$ is tame for $f^{-1}$.  We know that tame pairs are generic in $X\times\Aut\,X$ \cite[Proposition 1(a)]{AL2020}, but we do not know whether or not non-tame points actually exist.
 
\begin{remark}   \label{r:def-tame}
We note that $x\in\tam(f)\setminus(\rne(f)\cup\rne(f^{-1}))$ if and only if  there are sequences $x_n\to x$ in $X$, $f_n\to f$ in $\Aut\,X$, and integers $j_n, k_n\geq 0$, such that $f_n^{j_n}(x_n)\to \infty_X$ and $f_n^{-k_n}(x_n) \to \infty_X$ as $n\to\infty$.  Here $\infty_X$ denotes the point at infinity in the one-point compactification of $X$.
\end{remark}

\begin{lemma}  \label{l:creating-cycles}
Let $f$ be an automorphism of a Stein manifold $X$ with the density property.  Let $K\subset X$ be compact, let $p\in\tam(f)\setminus(\rne(f)\cup\rne(f^{-1}))$, and let $W$ be a neighbourhood of $f$ in $\Aut(X)$.  Then:
\begin{enumerate}
\item[(a)]  There is $h\in W$ such that $p\in\att(h)$ and the $h$-orbit of $p$ leaves $K$.
\item[(b)]  There is $h\in W$ such that $p\in\rep(h)$ and the $h$-orbit of $p$ leaves $K$.
\item[(c)]  There is $h\in W$ such that $p\in\sad(h)$ and the $h$-orbit of $p$ leaves $K$.
\end{enumerate}
\end{lemma}
The analogue for endomorphisms of an Oka-Stein manifold of this result is \cite[Theorem 3]{AL2022}.

\begin{proof}
We may assume that every automorphism close enough to $f$ on $K$ lies in $W$.  Choose a holomorphically convex compact set $L$ such that $p\in L^\circ$ and $f(K)\subset L^\circ$, so that $g(K)\subset L$ for every automorphism $g$ close enough to $f$ on $K$.

By Remark \ref{r:def-tame}, there are $g\in\Aut\, X$ and $q\in X$ arbitrarily close to $f$ and $p$ respectively, such that the forward $g$-orbit of $q$ is not contained in $L$, say $g^n(q)\notin L$, with $n\geq 1$ as small as possible, and the backward $g$-orbit of $q$ is not contained in $L$, say $g^{-m}(q)\notin L$, with $m\geq 1$ as small as possible.

By \cite[Theorem 2]{Varolin2000}, there is $\phi\in\Aut\, X$ as close to the identity as we wish on $L$, such that $\phi$ fixes $g^{-m}(q),\ldots, g^{n-1}(q)$, and $\phi(g^n(q))=g^{-m}(q)$.  Furthermore, we may take $\phi$ to have any invertible derivative at $g^n(q)$.  Then $\phi\circ g$ has $q$ as a periodic point and $\phi\circ g$ is as close to $g$ as we wish on $g^{-1}(L)\supset K$.  Since $X$ has the density property, it has automorphisms arbitrarily close to the identity that interchange $p$ and any sufficiently nearby point.  We conjugate $\phi\circ g$ by such an automorphism to obtain the desired automorphism $h$.
\end{proof}

Using Lemma \ref{l:creating-cycles}, we prove the following analogue of \cite[Theorem 2]{AL2022}.

\begin{lemma}  \label{t:closures-of-periodic}
For a generic automorphism $f$ of a Stein manifold $X$ with the density property, 
\[ \tam(f)\setminus(\rne(f)\cup\rne(f^{-1})) \subset \overline{\att(f)}\cap\overline{\rep(f)}\cap\overline{\sad(f)}. \]
\end{lemma}

\begin{proof}
Let $\{U_n:n\geq 1\}$ be a countable basis for the topology of $X$.  Let $S_n$ be the open set of all $f\in\Aut\,X$ such that $f$ has an attracting cycle intersecting $U_n$.  Then $G=\bigcap \Aut\,X\setminus \partial S_n$ is a residual subset of $\Aut\,X$.  We will show that if $f\in G$ and $p\in \tam(f)\setminus(\rne(f)\cup\rne(f^{-1}))$, then $p\in\overline{\att(f)}$.

Let $f\in G$.  If we show that for all $n$ such that  $p\in U_n$ we have  $f\in S_n$, it immediately follows that  $p\in\overline{\att(f)}$.  By definition of $G$, it suffices to show that $f\in \overline S_n$, which is immediate from Lemma \ref{l:creating-cycles}.  The proofs for $\rep(f)$ and $\sad(f)$ are analogous.
\end{proof}

Let $\Lambda$ be a subset of a topological manifold and let $f: \Lambda\to \Lambda$ be continuous.  The map $f$ is said to be {\it chaotic} (in the sense of Touhey) if for every two nonempty open subsets of $\Lambda$, there is a cycle that visits both of them.  We say that the map is {\it Devaney-chaotic} on $\Lambda$ if the periodic points of $f$ are dense in $\Lambda$ and there is a point in $\Lambda$ whose forward orbit is dense.  See \cite{BBCDS1992} and \cite{Touhey1997}.  In general, if $f$ is chaotic, then $f$ is Devaney-chaotic, and if $\Lambda$ is perfect, then the two notions are equivalent.

\begin{theorem}  \label{t:chaotic-and-non-chaotic-Julia-sets}
For a generic automorphism $f$ of a Stein manifold $X$ with the density property, the following hold.
\begin{enumerate}
\item[(a)]  $J_f^*$ is not compact.
\item[(b)]  $J_f^*=J_f\cap\tam(f)$.
\item[(c)]  $J_f^* \subset \overline{\att(f)} \cap \overline{\rep(f)}$.
\item[(d)] $f$ is chaotic on $J_f^*$.
\item[(e)]  $J_f^*$ is the largest  forward invariant  subset of $X$ on which $f$ is chaotic.
\item[(f)]   $J_f^*$ is perfect. 
\item[(g)]  The non-chaotic Julia set $M_f$ is completely invariant, contains no periodic points, and the total orbit of every point in $M_f$ is relatively compact.
\end{enumerate}
\end{theorem}

Note that by (a), $J_f^*$ is not empty.  We do not know whether or not $M_f$ is empty for a generic automorphism $f$.

\begin{proof}
(a)  That $J_f^*$ is not compact for generic $f$ is immediate from Remark \ref{r:infinitely-many}.

(b)  By Lemma \ref{t:closures-of-periodic}, $J_f\cap\tam(f) \subset \overline{\sad(f)}$.  Conversely, we clearly have $\overline{\sad(f)}\subset J_f$.  We will sketch a proof that $\overline{\sad(f)}\subset \tam(f)$.  Let $x\in\overline{\sad(f)}$ and take $q\in\sad(f)$ close to $x$.  Choose $y\in W_f^s(q)$ and $z\in W_f^u(q)$ outside a big compact set.  By the lambda lemma, there are points $y'$ close to $y$, $x'$ close to $q$, and $z'$ close to $z$ such that the forward $f$-orbit of $y'$ contains $x'$ and $z'$ (in this order).  With such points $x'$ we can form a sequence $(x_n)$ with $x_n\to x$ such that there are integers $j_n, k_n\geq 0$ with $f^{j_n}(x_n)\to\infty_X$ and $f^{-k_n}(x_n)\to\infty_X$.  This shows that $x\in\tam(f)$ (see Remark \ref{r:def-tame}).

(c) Lemma \ref{t:closures-of-periodic} together with (b) imply that  $J_f^* \subset \overline{\att(f)} \cap \overline{\rep(f)}$.

(d)  Take a countable basis $\{U_n : n\geq 1\}$ for the topology of $X$.  Let $S_{m,n}$ be the open subset of $\Aut\,X$ of automorphisms with a saddle cycle through $U_m$ and $U_n$.  Then $G=\bigcap \Aut\,X\setminus \partial S_{m,n}$ is a residual subset of $\Aut\,X$ and we claim that every $f\in G$ is chaotic on $J_f^*$.  So let $f\in G$ and take $U_m$ and $U_n$ both intersecting $J_f^*$.  We need $f\in S_{m,n}$, that is, we need a saddle cycle for $f$ through $U_m$ and $U_n$ (notice that the saddle cycle is necessarily contained in $J_f^*$).   By the definition of $G$, it suffices to show that $f\in \overline S_{m,n}$, which follows as in \cite[Section~5]{AL2019}.

(e)  By Theorem \ref{t:Kupka-Smale-1}, every periodic point of $f$ is hyperbolic.  Let $\Lambda$ be a forward invariant subset of $X$ on which $f$ is chaotic.  Then $\Lambda$ lies in the closure of the set of periodic points of $f$.  If there is an attracting cycle in $\Lambda$, then clearly $f$ cannot be chaotic on $\Lambda$.  Since $f$ is invertible, the same is true if there is a repelling cycle in $\Lambda$. Thus $\Lambda\subset \overline{\sad(f)}$.

(f) If $J_f^*=\overline{\sad(f)}$ had an isolated point, it would be a saddle point.  Being non-compact, $J_f^*$ could not consist of the cycle of that saddle point alone. But then $f$ would not be chaotic on $J_f^*$.

(g)  Since $J_f$ and $J_f^*$ are completely invariant, so is $M_f$.  Attracting and repelling periodic points lie outside $J_f$ and saddle periodic points lie in $J_f^*$, so none lie in $M_f$ by Theorem \ref{t:Kupka-Smale-1}.  Suppose that the total orbit of $x\in J_f$ is not relatively compact.  Say the forward orbit of $x$ is not relatively compact.  Let $U$ be a neighbourhood of $x$ and $K\subset X$ be compact.  We will show that there is $y\in U$ such that neither the forward nor the backward orbit of $y$ is contained in $K$.  This implies that $x$ is tame.  Find $m\geq 1$ such that $f^m(x)\notin K$ and choose a neighbourhood $V\subset U$ of $x$ such that $f^m(V)\cap K=\varnothing$.  Since $x\notin F_f^-$, by Montel's theorem there is $y\in V$ whose backward orbit is not contained in $K$; neither is the forward orbit of $y$.
\end{proof}

\section{Homoclinic points} 
\label{sec:homoclinic}

\noindent
First of all, note that homoclinic and heteroclinic points of saddle periodic points lie in the Julia set.

\begin{proposition}   \label{p:homoclinic-are-tame}
Transverse homoclinic and heteroclinic points of an automorphism $f$ of a Stein manifold $X$ are tame.
\end{proposition}

\begin{proof}
Let $x$ and $y$ be saddle periodic points, take $p\in W_f^s(x)\cap W_f^u(y)$, and let $K\subset X$ be compact.  We show that there are orbits of $f$ that come arbitrarily close to $p$ and leave $K$ in forward and backward time.  (The definition of tameness would allow us to perturb $f$, but it is not necessary.)  By passing to an iterate of $f$, we can assume that $x$ and $y$ are fixed.

Take $z\in W_f^s(y)\setminus K$ and a small polydisc $D_1$ transverse to $W_f^s(y)$ centred at $z$.  As a polydisc $D_2$ transverse to $W_f^u(y)$ centred at $p$, choose simply a small portion of $ W_f^s(x)$ around $p$.  By the lambda lemma, there is an orbit starting in $D_1$ and ending in a point $p' \in D_2 \subset W_f^s(x)$.  Choose a small polydisc $D_3$ transverse to $W_f^s(x)$ centred at $p'$.  Choose also $w\in W_f^u(x)\setminus K$ and a small polydisc $D_4$ transverse to $W_f^u(x)$ centred at $w$.  By the lambda lemma, there is an orbit starting at a point $p''\in D_3$ and ending in $D_4$.  If $D_3$ is small enough, $p''$ will be close enough to $p'$ that the backward orbit of $p''$ leaves $K$.
\end{proof}

The following result is an immediate consequence of the proposition, Lemma~\ref{l:stable-in-julia}, and Theorem \ref{t:chaotic-and-non-chaotic-Julia-sets}.

\begin{corollary}   \label{c:transverse-homoclinic-are-in-chaotic-Julia}
For a generic automorphism $f$ of a Stein manifold $X$ with the density property, all transverse homoclinic and heteroclinic points of saddle periodic points of $f$ lie in the chaotic Julia set $J_f^*$ of $f$.
\end{corollary}

In the proof of the next result, we produce transverse homoclinic points.

\begin{theorem}\label{t:homoclinic-perturbation}
Let $X$ be a Stein manifold with the density property.  Let $p$ be a saddle periodic point of $f\in\Aut\,X$.  Then every neighbourhood $W$ of $f$ in $\Aut\,X$ contains an automorphism with $p$ as a saddle periodic point with a transverse homoclinic point. 
\end{theorem}

\begin{proof}
We prove the theorem for a saddle fixed point.  The case of a periodic point is similar.  Let $d$ be a distance inducing the topology on $X$.  Find $\epsilon>0$ and a compact subset $K\subset X$ such that 
\[ \{g\in {\rm Aut}(X): d_K(f,g)<\epsilon\}\subset W. \]
We may assume that the local stable and unstable manifolds $\Gamma^s_f(p,r)$ and  $\Gamma^u_f(p,r)$ are contained in a coordinate ball centred at $p$.  Let $r>0$ be given by Lemma \ref{l:delicate}.  By enlarging $K$ if necessary, we may assume that $K$ is holomorphically convex and contains the polydisc $\overline{\Delta^n(0,r)}$ in its interior.
 
Let $x_0 \in W_f^s(p)\setminus K$ such that $f^n(x_0) \in K$ for all $n\geq 1$.  Let $y_0\in W_f^u(p)\setminus K$ such that $f^{-n}(y_0)\in K$ for all $n\geq 1$.  Find $n_s\geq 1$ such that $f^{n_s}(x_0)\in \Gamma^s_f(p,r)$ and $n_u\geq 1$ such that $f^{-n_u}(y_0)\in \Gamma^u_f(p,r)$.  

Let $U$ be a Runge neighbourhood of $K$ containing neither $x_0$ nor $y_0$.  Let $V$ be a neighbourhood of $y_0$ and let $\varphi : [0,1]\times V\to X$ be a $C^1$ isotopy such that for all $t\in [0,1]$,
\begin{enumerate}
\item $\varphi_t : V\to X$ is holomorphic and injective,
\item $\varphi_t(V)$ is disjoint from $U$, 
\item $U\cup \varphi_t(V)$ is Runge,
\item $\varphi_0$ is the inclusion of $V$ into $X$, 
\item $\varphi_1(y_0)=x_0$ and $d_{y_0}\varphi_1(T_{y_0}W_f^u(p))$ is transverse to $T_{x_0}W_f^s(p).$
\end{enumerate}

By the Anders\'en-Lempert theorem \cite[Theorem 4.10.5]{Forstneric2017}, there is a sequence $(\Phi_j)$ in $\Aut\,X$ such that $\Phi_j\to \id$ on $U$ and $\Phi_j\to \varphi_1$ on $V$, uniformly on compact subsets.  Let $ f_j:= f\circ \Phi_j$. 

Let $\psi:D\to \Gamma^u_f(p,r)$ be a holomorphic graph parametrisation of the local unstable manifold  $\Gamma^u_f(p,r)$ near $f^{-n_u}(y_0)$, defined on a small polydisc $D$. If $D$ is small enough, then $(f^{n_u}\circ \psi)(D)$ is a holomorphically embedded piece of the unstable manifold $W^u_f(p)$ containing $y_0$ and contained in the neighbourhood $V$. Then $(\varphi_1\circ f^{n_u}\circ \psi)(D)$ is an embedded complex submanifold which intersects the stable manifold $W^u_f(p)$ transversally at $x_0$.  Finally, $(f^{n_s}\circ \varphi_1\circ f^{n_u}\circ \psi)(D)$ is an embedded complex submanifold which intersects the local stable manifold $\Gamma^s_f(p,r)$ transversally at $f^{n_s}(x_0)$.

 For large $j$, let $p_j:=\eta(f_j)$ be the saddle fixed point near $p$ given by Lemma \ref{l:saddlepert}.  Then by Lemma \ref{l:delicate}, there is a  holomorphic graph parametrisation $\psi_j:D\to \Gamma^u_{ f_j}(p_j,r)$ which converges uniformly to $\psi$ as $j\to\infty$.  It follows that the holomorphic map $ f_j^{n_s+n_u}\circ \psi_j$ converges uniformly to $f^{n_s}\circ \varphi_1\circ f^{n_u}\circ \psi$, so if $j$ is large enough, the embedded complex submanifold $ (f_j^{n_s+n_u}\circ \psi)(D)$ intersects the 
 local stable manifold $\Gamma^s_{f_j}(p_j,r)$ transversally in a homoclinic point for $f_j$.  Conjugation by a small perturbation of the identity makes the point $p_j$ coincide with $p$.
\end{proof}

\begin{corollary}
For a generic automorphism $f$ of a Stein manifold $X$ with the density property, the chaotic Julia set $J_f^*$ is the closure of the subset of transverse homoclinic points.
\end{corollary}

\begin{proof}
Let $\{U_n:n\geq 1\}$ be a countable basis for the topology on $X$.  Let $S_n\subset \Aut\,X$ be the set of automorphisms that have a saddle periodic point in $U_n$ whose stable and unstable manifolds intersect transversely.  By Lemma \ref{l:delicate2}, $S_n$ is open, so the subset $G:=\bigcap\Aut\,X \setminus \partial S_n$ of $\Aut\,X$ is residual.

Let $f\in G$ and take $n\geq 0$ with $U_n\cap J_f^*\neq \varnothing$.  By the definition of $J_f^*$, $f$ has a saddle periodic point in $U_n$, so by Theorem \ref{t:homoclinic-perturbation}, $f\in \overline S_n$.  Also, by the definition of $G$, $f\not\in \partial S_n$, so $f\in S_n$.  Let $p\in U_n$ be a saddle periodic point of $f$ with a transverse homoclinic point $q$.  Then, if $m\geq 0$ is large enough, $f^m(q)$ is also a transverse homoclinic point and is contained in $U_n$.  
\end{proof}

By Theorem \ref{t:Kupka-Smale-2}, the corollary holds with the word \lq\lq transverse\rq\rq\ removed.

\section{Closing lemmas:  Open problems} 
\label{sec:closing}

\noindent
As before, we take $X$ to be a Stein manifold with the density property.   We say that the {\it closing lemma} holds for automorphisms of $X$ if, whenever $p\in X$ is a non-wandering point of an automorphism $f$ of $X$, every neighbourhood of $f$ in $\Aut\,X$ contains an automorphism of which $p$ is a periodic point.  Requiring $p$ to be hyperbolic results in an equivalent statement by the perturbation lemma \cite[Lemma 1]{AL2020}.  The {\it generic closing lemma} requires this to hold for automorphisms $f$ in a suitable residual subset of $\Aut\,X$.  The {\it generic density theorem} states that hyperbolic periodic points are dense in the non-wandering set of a generic automorphism of $X$.  It is usually proved as a consequence of the closing lemma, but an inspection of the standard proof (as in \cite{AL2020}) easily shows that the generic closing lemma suffices.  Conversely, conjugation by a suitable perturbation of the identity (provided by \cite[Lemma 3.2]{Varolin2000}; see also \cite[Proposition 1]{AL2019}) shows that the generic density theorem implies the generic closing lemma.  It is an open question whether the three statements are true, but with the non-wandering set replaced by the tame non-wandering set, they were proved in \cite{AL2020}.  The closing lemma is only in question when $p\notin\rne(f)\cup\rne(f^{-1})$, because otherwise $p$ is an attracting or repelling periodic point of $f$ itself (see Step 1 of the proof of \cite[Theorem 2]{AL2020}).

We introduce a variant of the  closing lemma and say that the {\it large-cycles closing lemma} holds for automorphisms of $X$ if, whenever $p\in X$ is a non-wandering point of an automorphism $f$ of $X$ with $p\notin\rne(f)\cup\rne(f^{-1})$ and $K$ is a compact subset of $X$, every neighbourhood of $f$ in $\Aut\,X$ contains an automorphism of which $p$ is a periodic point such that the orbit of $p$ does not lie in $K$.  Clearly the  large-cycles closing lemma implies the closing lemma.  The {\it generic large-cycles closing lemma} requires this to hold for automorphisms $f$ in a suitable residual subset of $\Aut\,X$ (recall that generically $\Omega_f \setminus \big(\rne(f)\cup\rne(f^{-1})\big)=J_f$).  It is equivalent to the {\it large-cycles generic density theorem}, which states that for a generic automorphism $f$ of $X$ and any compact subset $K$ of $X$, periodic points whose orbits do not lie in $K$ are dense in $ J_f$.

\begin{theorem}  \label{t:equality-of-Julia-sets}
Let $X$ be a Stein manifold with the density property.  The following are equivalent.
\begin{enumerate}
\item[(i)]  The generic large-cycles closing lemma for automorphisms of $X$.
\item[(ii)]  The large-cycles generic density theorem.
\item[(iii)]  Every non-wandering point of a generic automorphism of $X$ is tame.
\item[(iv)]  $J_f^* = J_f$ for a generic automorphism $f$ of $X$.
\item[(v)]  $f$ is chaotic on $J_f$ for a generic automorphism $f$ of $X$.
\item[(vi)]  Periodic points are dense in $J_f$ for a generic automorphism $f$ of $X$.
\end{enumerate}
\end{theorem}

\begin{proof}
It is easy to see that (i) and (ii) are equivalent and that (iii)--(vi) are equivalent.  Also (ii) implies (iii).  It remains to show that (v) implies (ii).  Since $J_f$ is not compact, there is $y\in J_f\setminus K$.  Let $x\in J_f$.  Since $f$ is chaotic on $J_f$, there is a cycle visiting both an arbitrary neighbourhood of $x$ and an arbitrary neighbourhood of $y$. 
\end{proof}

\begin{remark}
There is more structure inside the Julia set.  For an automorphism $f$ of $X$, let $I_f=\overline{\att(f)}\setminus\att(f)$ and $I'_f=I_{f^{-1}}=\overline{\rep(f)}\setminus\rep(f)$.  Clearly, $I_f$ and $I'_f$ are closed and completely invariant.  Also, $I_f\cup I'_f \subset J_f$ for generic $f$.  Indeed, $\overline{\att(f)}\subset\Omega_f$ since $\Omega_f$ is closed, and by Theorem \ref{t:non-wand-equals-Julia}, $\Omega_f = J_f \cup \att(f) \cup \rep(f)$, so $I_f\subset J_f$.  (Note that attracting periodic points of an automorphism cannot accumulate on a repelling periodic point.)  Similarly, $I'_f\subset J_f$.  Moreover, by Theorem \ref{t:chaotic-and-non-chaotic-Julia-sets}, $J_f^*\subset \overline{\att(f)}\cap \overline{\rep(f)}= I_f\cap I'_f$ for generic $f$.  The generic closing lemma says precisely that $J_f=I_f\cup I'_f$ for generic $f$.  The generic large-cycles closing lemma says precisely that all six completely invariant sets are the same.
\[ \xymatrix{  &  &  I_f \ar@{^{(}->}[dr] & & & & \\
J_f^* \ar@{^{(}->}[r] & I_f\cap I'_f  \ \ar@{^{(}->}[dr] \ar@{^{(}->}[ur] & & I_f\cup I'_f \ar@{^{(}->}[r] & J_f  \\
&  &  I'_f \ar@{^{(}->}[ur] & & & &
 } \]
\end{remark}

\end{document}